\documentclass[12pt,a4paper]{article}
\usepackage[a4paper,top=2.5cm,bottom=2.5cm,left=2.5cm,right=2.5cm]{geometry}
\usepackage{amsmath}
 \usepackage{graphicx}
\usepackage{amsthm}
\usepackage{color}
\usepackage{fancyhdr}
\usepackage{pdfpages}
\usepackage{amssymb}
\usepackage{paralist}
\usepackage{extarrows}
\usepackage{setspace}

\usepackage{braket} 
\usepackage[colorlinks=true,citecolor=blue,linkcolor=blue,anchorcolor=blue,urlcolor=blue]{hyperref}
\usepackage{float}
\allowdisplaybreaks
\numberwithin{equation}{section}
\newtheorem{theorem}{Theorem}[section]
\newtheorem{lemma}[theorem]{Lemma}
\newtheorem{proposition}[theorem]{Proposition}

\newtheorem{definition}[theorem]{Definition}
\begin{document}
\title{The cone Moser-Trudinger  inequalities and their applications }
\author{Fei Fang \footnote{ }\\  \small  \emph{\small Department of Mathematics,
Beijing Technology and Business University }\\ \emph{ \small  Beijing, 100048, China} \\   Chao Ji \footnote{Corresponding author. E-mail address:  jichao@ecust.edu.cn(C. Ji).}  \\ \small  \emph{Department of Mathematics, East China University of Science and Technology}\\  \small \emph{Shanghai, 200237, China}}
\maketitle

\noindent \textbf{\textbf{Abstract:}}
 In this article, we firstly study the cone Moser-Trudinger  inequalities and their best exponents $\alpha_2$ on both bounded  and unbounded domains $\mathbb{R}^2_{+}$.
Then, using the cone Moser-Trudinger  inequalities, we study the existence of weak solutions to
the nonlinear  equation
\begin{equation*}
\left\{
\begin{array}{ll}
-\Delta_{\mathbb{B}} u=f(x,  u), &\mbox{in}\  x\in \mbox{int} (\mathbb{B}),    \\
u= 0,  &\mbox{on}\  \partial\mathbb{B},
\end{array}
\right.
\end{equation*}
  where $\Delta_{\mathbb{B}}$ is Fuchsian type Laplace operator investigated
with totally characteristic degeneracy on the boundary $x_1 =0$, and  the nonlinearity $f$ has the subcritical exponential growth or the critical exponential growth.


\section{Introduction}
Let $\Omega\subset \mathbb{R}^N$ be an open set, $N\geq 2$.  It is well-known that $W_0^{1,p}(\Omega)\subset L^{\frac{Np}{N-p}}$ if $1\leq p<N$, and
$W_0^{1,p}(\Omega) \subset L^{\infty}(\Omega)$, if $p>N$. The case $p=N$ is the limit case of these imbeddings and it is known that
$W_0^{1,N}(\Omega) \subset L^{q}(\Omega)$ for $N\leq q<\infty$ and $W_0^{1,N}(\Omega) \not\subset L^{\infty}(\Omega)$.

Trudinger \cite{Tr} and Pohozaev \cite{Pohozaev}  found independently that the maximal
growth is of exponential type. More precisely, there exist two positive constants $\alpha$ and $C$ depending
only on $N$ such that
\begin{equation}\label{e39}
\int_{\Omega} e^{\alpha\left(\frac{|u(x)|}{|\nabla  u(x)|_N}\right)^{\frac{N}{N-1}}}dx\leq C|\Omega|
\end{equation}
for $u\in W_0^{1,N}(\Omega)\setminus\{0\}$, where the constants $\alpha, C$ are independent of $u$ and $\Omega$. In order to prove \eqref{e39}, Trudinger in \cite{Tr} used a combination of the power series expansion of the exponential function and sharp multiplicative inequalities
$$|u|_q\leq C(N,q)|u|_N^{\frac{N}{q}}|\nabla u|_N^{1-\frac{N}{q}}.$$

There are many types of extensions for the Trudinger-Moser inequality.  The first one is to find the
best exponents in (\ref{e39}). Moser \cite{Mo} (see also \cite{Len} )  showed  that (\ref{e39}) holds for $\alpha\leq \alpha_N$ but not for $\alpha>\alpha_N$,
where $$\alpha_N=N\omega_{N-1}^{\frac{1}{N-1}}$$ and $\omega_{N-1}$ is the surface area of the unit sphere in $\mathbb{R}^N$.  Moser used  symmetrization of functions and reduce (\ref{e39}) to one-dimensional inequality. And the reader can be  referred to \cite{Car,Flu, Mcl, Str} for the attainability of

\begin{equation}\label{e40}
\sup\left\{\int_{\Omega} e^{\alpha_N\left(\frac{|u(x)|}{|\nabla  u(x)|_N}\right)^{\frac{N}{N-1}}}dx: u\in W_0^{1,N}(\Omega)\setminus\{0\}\right\}.
\end{equation}

The second direction is to extend Trudinger's result for unbounded domains and
for Sobolev spaces of higher order and fractional order (see \cite{Adac,Ada,Cao,Oga,Oza,Stri}).
In  \cite{AF,Oga,Oza}, the following Trudinger Type inequality was studied without  best exponents

\begin{equation}\label{e41}
\int_{\mathbb{R}^N}\left( e^{\alpha\left(\frac{|u(x)|}{|\nabla  u(x)|_N}\right)^{\frac{N}{N-1}}}-\sum_{j=0}^{N-2}\frac{1 }{j!}\left(\alpha\left(\frac{|u(x)|}{|\nabla  u(x)|_N}\right)^{\frac{N}{N-1}}\right)^j\right)dx\leq C \frac{|u(x)|_N^N}{|\nabla  u(x)|_N^N}
\end{equation}
for $u\in W_0^{1,N}(\mathbb{R}^N)\setminus\{0\}$. In \cite{Adac}, the best exponents $\alpha$ in (\ref{e41})   was obtained, moreover, by Moser's idea,
a simplified proof for  (\ref{e41}) was given. With regard to the case of higher order derivatives, since the symmetrization is not available, D. Adams \cite{Ada}
 proposed a new idea to find the sharp constants for higher order Moser’s type
inequality, that is, to express $u$ as the Riesz potential of its gradient of order $m$, and
then apply O'Neil's result on the rearrangement of convolution functions and use
techniques of symmetric decreasing rearrangements.

In \cite{zhang}, the authors proved an   affine Moser-Trudinger inequality. The authors of  \cite{luguozhen}  proved the
sharp singular affine Moser-Trudinger inequalities on both bounded and unbounded domains in $\mathbb{R}^N$ and they
improved Adams type inequality in the spirit of Lions \cite{Lions}.

Another extension is to establish the Trudinger-Moser inequality and the Adams inequality
on compact Riemannian manifolds and noncompact Riemannian manifolds  (see \cite{Lixy, Yang} ).

Moser-Trudinger inequalities have played important roles and have
been widely used in geometric analysis and PDEs, see for example  \cite{chang,do,Fig,LamLu,tian}, 
and references therein.
The main purpose of this paper is to study the cone Moser-Trudinger  inequalities and their applications. To the best of our knowledge,
the related research is rare and

The outline of this paper are as follows. In Section 2 we give some preliminaries, such as the definition of the cone Sobolev spaces and
some lemmas which will be used in the later sections. In Section 3, we give the cone Moser-Trudinger  inequalities and their proofs. In Section 4, as the applications of the cone Moser-Trudinger inequalities,  the existence of multiple solutions to the degenerate elliptic equations
with the subcritical exponential growth or the critical exponential growth will be discussed. Our main results are Theorem \ref{t0}, Theorem \ref{t2},  Theorem \ref{t3}, Theorem \ref{t41} and Theorem \ref{t42}.

\section{Cone Sobolev spaces}
In this section we introduce the manifold with conical singularities and the corresponding
cone Sobolev spaces.

Let $X$ be a closed, compact, $C^{\infty}$ manifold. We set $X^{\triangle}=(\bar{\mathbb{R}}_{+}\times X)/(\{0\}\times X)$
as a local model interpreted as a cone with the base $X$.  Next, we denote
$X^{\wedge}=\mathbb{R}_{+}\times X$  as the corresponding open stretched cone
with the base $X$.

A $n$-dimensional manifold $B$ with conical singularities is a topological space with a finite subset $B_0=\{b_1,\cdots, b_M\}\subset B$ of
conical singularities, with the following two properties:

(1) $B\setminus B_0$ is a $C^{\infty}$ manifold.

(2) Each $b\in B_0$ has an open  neighbourhood $U$ in $B$ such that there is a homeomorphism $\phi: U\rightarrow X^{\triangle}$ for some closed
compact $C^{\infty}$ manifold $X=X(b)$, and $\phi$ restricts to a  diffeomorphism  $\phi': U\setminus\{b\} \rightarrow X^{\wedge}$.

For such a manifold, let $n \geq 2$  and $X\subset S^{n-1}$ be a bounded open set in the unit sphere of $\mathbb{R}^n_{x}$.
The set $B:=\{x \in \mathbb{R}^n\setminus\{0\}:  \frac{x}{|x|}\in X\} \cup\{0\}$ is an infinite cone with the base $X$
 and the conical point $\{0\}$. Using the polar coordinates, one can get a description of $B\setminus\{0\}$
  in the form $X^{\wedge}=\mathbb{R}^{+}\times X$, which is called the open stretched cone with the base $X$,
  and $\{0\} \times X$ is the boundary of $X^{\wedge}$.

Now, we assume that the manifold $B$ is paracompact and of dimension $n$.
By this assumption we can define the stretched manifold associated with $B$.
Let $\mathbb{B}$ be a $C^{\infty}$ manifold with compact $C^{\infty}$ boundary $\partial \mathbb{B}=\cup_{x\in B_0} X(x)$
for which there exists a diffeomorphism $B\setminus B_0=\mathbb{B}\setminus\partial \mathbb{B}:=`int \mathbb{B}$, the restriction of which
to $G_1\setminus B_0=U_1\setminus \partial \mathbb{B}$ for an open neighborhood $G_1\subset B$  near the points of $B_0$ and a collar neighborhood
$U_1\subset \mathbb{B}$ with $U_1=\cup_{x\in B_0} \{[0, 1)\times X(x)\}$.

The typical differential operators on a manifold with conical singularities, called Fuchs type, are operators that are in a neighborhood of $x_1=0$ of the following form
$$A=x_1^{-m} \sum_{k=0}^{m} a_k(x_1)(-x_1\partial_{x_1})^{k}$$
with $(x_1, x) \in X^{\wedge}$ and $a_k(x_1)\in C^{\infty} (\bar{\mathbb{R}}_{+}, \mbox{Diff}^{m-k}(X))$ (see  \cite{es,mali,sb}).
The differential $x_1\partial x_1$ in Fuchs type operators provokes us to apply the Mellin transform $M$ (see Definition \ref{d1}).
\begin{definition}\label{d1}
Let $u(t)\in C_0^{\infty}(\mathbb{R}_{+})$, $Z\in \mathbb{C}$. The Mellin transform is defined by the formula
$$Mu(z)=\int_0^{+\infty} t^zu(t)\frac{dt}{t},$$and
$$M: C_0^{\infty}(\mathbb{R}_{+})\rightarrow \mathcal{A}(C),$$
where $\mathcal{A}(C)$ denotes the spaces of entire functions.
\end{definition}

\begin{proposition}\label{p1}
 The Mellin transform satisfies the following identities
 \begin{itemize} \addtolength{\itemsep}{-1.5 em} \setlength{\itemsep}{-5pt}
   \item [\emph{(1)}] $M((-t\partial_t)u)(z)=zM(z)$,
   \item  [\emph{(2)}] $M(t^{-p}u)(z)=(Mu)(z-p)$,
   \item [\emph{(3)}] $M((\log t)u)(z)=(\partial_{z}Mu)(z)$,
   \item [\emph{(4)}] $M(u(t^{\beta}))(z)=\beta^{-1}(Mu)(\beta^{-1}z)$,
 \end{itemize}
 for $t\in \mathbb{R}_{+},$ $z,p\in \mathbb{C}$, $\beta\in \mathbb{R}\setminus \{0\}$, and $u\in C_0^{\infty}(\mathbb{R}_{+})$.
\end{proposition}

Now let $\Gamma_{\beta}=\{z\in \mathbb{C}: \mbox{Re}\  z=\beta\}$. We define the weighted Mellin transform with weight data $\gamma$
as follows
$$M_{\gamma}u:=Mu|_{\Gamma_{\frac{1}{2}-\gamma}}=\int_0^{\infty} t^{\frac{1}{2}-\gamma+i\tau}u(t)\frac{dt}{t},$$
and the inverse weighted Mellin transform is defined as

$$(M_{\gamma}^{-1}g)(t)=\frac{1}{2\pi i}\int_{\gamma_{\frac{1}{2}-\gamma}}t^{-z}g(z)dz.$$
For $u(t)\in C_0^{\infty}(\mathbb{R}_{+})$, set $S_{\gamma}u(r)=e^{-\left(\frac{1}{2}-\gamma\right)r}u(e^{-r})$, then we obtain
\begin{equation}\label{e05}
  (M_{\gamma}u)\left(\frac{1}{2}-\gamma+i\tau\right)=(\mathcal{F}S_{\gamma u}(\tau)),
\end{equation}
where $\mathcal{F}$ is the 1-dimensional Fourier transform corresponding to $t$. Indeed, by changing variables $t=e^{-r}$ and set $z=\frac{1}{2}-\gamma+i\tau\in \mathbb{C}$, it is easy to see that
\begin{align*}
   (\mathcal{F}S_{\gamma}u)(\tau)&=\int_{-\infty}^{+\infty} e^{-ir\tau} e^{-\left(\frac{1}{2}-\gamma\right)r}u(e^{-r})dr=\int_{-\infty}^{+\infty} e^{-\left(\frac{1}{2}-\gamma+-i\tau\right)r}u(e^{-r})dr \\
   & =\int_{0}^{+\infty} t^{z}u(t)\frac{dt}{t}=(M_{\gamma}u)\left(\frac{1}{2}-\gamma+i\tau\right).
\end{align*}

Then, we  have the following result.

\begin{lemma}[\cite{sb}]\label{l1}
The operator $M_{\gamma}: C_0^{\infty}(\mathbb{R}_{+})\rightarrow S(\Gamma_{\frac{1}{2}-\gamma})$ extends by continuity
to an isomorphism
$$M_{\gamma}: L_2^{\gamma}(\mathbb{R}_{+})\rightarrow L^2(\Gamma_{\frac{1}{2}-\gamma})$$
for all $\gamma\in \mathbb{R}$ and $ L_2^{\gamma}(\mathbb{R}_{+})=t^{\gamma}  L_2(\mathbb{R}_{+})$, where
$$|u|_{ L_2^{\gamma}(\mathbb{R}_{+})}=(2\pi)^{-\frac{1}{2}}|M_{\gamma}u|_{L^2\left(\Gamma_{\frac{1}{2}-\gamma}\right)}.$$
\end{lemma}

\begin{definition}\label{d2}
For $s,\gamma\in \mathbb{R}$, we denote by $\mathcal{H}_2^{s,\gamma}(\mathbb{R}_{+}^{n+1})$ the space of all $u\in \mathcal{D}'(\mathbb{R}_{+}^{n+1})$ such that
$$\frac{1}{2\pi i}\int_{\Gamma_{\frac{n+1}{2}-\gamma}}\int_{\mathbb{R}^n}(1+|z|^2+|\xi|^2)^s\left|\left(M_{\gamma-\frac{n+1}{2},t\rightarrow z}\mathcal{F}_{x\rightarrow\xi}u\right)(z,\xi)\right|^2dzd\xi<+\infty,$$
where $M_{\gamma-\frac{n+1}{2}}$ is the weighted Mellin transform and $\mathcal{F}_{x\rightarrow\xi}$ the $n$-dimensional Fourier transform. Naturally, the
space $\mathcal{H}_2^{s,\gamma}(\mathbb{R}_{+}^{n+1})$ admits a norm
$$\|u\|_{\mathcal{H}_2^{s,\gamma}(\mathbb{R}_{+}^{n+1})}=\left[\frac{1}{2\pi i}\int_{\Gamma_{\frac{n+1}{2}-\gamma}}\int_{\mathbb{R}^n}(1+|z|^2+|\xi|^2)^s\left|\left(M_{\gamma-\frac{n+1}{2},t\rightarrow z}\mathcal{F}_{x\rightarrow\xi}u\right)(z,\xi)\right|^2dzd\xi\right]^{\frac{1}{2}}.$$
\end{definition}
Then we easily obtain  the weighted Mellin Sobolev space of integer smoothness. 

\begin{definition}\label{d3}
Let $L_2(\mathbb{R}_{+}^{n+1})$ be the space of square integrable functions on $\mathbb{R}_{+}^{n+1}$, with respect to $dtdx$, and
 $(t,x)\in \mathbb{R}_{+}\times \mathbb{R}^n$. For $m\in \mathbb{N}$, and $\gamma \in \mathbb{R}$, we define
 \begin{equation}\label{e01}
   \mathcal{H}_2^{m,\gamma}(\mathbb{R}_{+}^{n+1})=\left\{u\in D'\left(\mathbb{R}_{+}^{n+1}\right): (t\partial_t)^{\alpha}\partial_x^{\beta}u\in t^{\gamma-\frac{n}{2}}L^2\left(\mathbb{R}_{+}^{n+1},dtdx\right)\right\},
 \end{equation}
 for arbitrary  $\alpha\in \mathbb{N}, \beta \in  \mathbb{N}^n$, and $|\alpha|+|\beta|\leq m$. Then $\mathcal{H}_2^{m,\gamma}(\mathbb{R}_{+}^{n+1})$
 is a Hilbert  space with the norm
 $$\|u\|_{\mathcal{H}_2^{m,\gamma}(\mathbb{R}_{+}^{n+1})}=\sum_{|\alpha|+|\beta|\leq m}\left[\int_{\mathbb{R}_{+}\times \mathbb{R}^n}
 | t^{\frac{n}{2}-\gamma}(t\partial_t)^{\alpha}\partial_x^{\beta}u(t,x)|^2 dtdx\right]^{\frac{1}{2}}.$$
\end{definition}
If we denote by $L_2(\mathbb{R}_{+}^{n+1})$ the space of square integrable functions with respect to the
measure $\frac{dt}{t} dx$, we can write (\ref{e01}) as follows:
 \begin{equation}\label{e03}
   \mathcal{H}_2^{m,\gamma}(\mathbb{R}_{+}^{n+1})=\left\{u\in D'\left(\mathbb{R}_{+}^{n+1}\right): t^{\frac{n+1}{2}-\gamma}(t\partial_t)^{\alpha}\partial_x^{\beta}u\in L_2\left(\mathbb{R}_{+}^{n+1},\frac{dt}{t}dx\right)\right\},
 \end{equation}
for all  $\alpha\in \mathbb{N}, \beta \in  \mathbb{N}^n$, and $|\alpha|+|\beta|\leq m$.
Here $m \in \mathbb{N}$ is called the smoothness of
Sobolev spaces, and $\gamma \in \mathbb{R}$ the flatness of $t$-variable.
Next, we introduce a map
\begin{equation}\label{e04}
  \left(S_{\frac{n+1}{2},\gamma}u\right)(r,x)=e^{-\left(\frac{n+1}{2}-\gamma\right)r}u(e^{-r},x)
\end{equation}
for $u(t, x) \in C_0^{\infty}(\mathbb{R}^{n+1}_{+})$, which is a continuous map $S_{\frac{n+1}{2},\gamma}: C_0^{\infty}(\mathbb{R}^{n+1}_{+})\rightarrow C_0^{\infty}(\mathbb{R}^{n+1})$.

Similar to (\ref{e05}), we can extend (\ref{e04}) to an isomorphism
$$\left(S_{\frac{n+1}{2},\gamma}\right): \mathcal{H}_2^{m,\gamma}(\mathbb{R}_{+}^{n+1})\rightarrow H_2^{m}(\mathbb{R}^{n+1}),$$
this isomorphism can also be said
the norm  $\|u\|_{\mathcal{H}_2^{m,\gamma}(\mathbb{R}_{+}^{n+1})}$ is equivalent to the norm $\|S_{\frac{n+1}{2},\gamma} u\|_{ H_2^{m}(\mathbb{R}^{n+1})}$,
where $H_2^{m}(\mathbb{R}^{n+1})$ denotes the distribution space for $(r, x)\in \mathbb{R}^{n+1}$ such that
 \begin{equation}\label{e06}
   H_2^{m,\gamma}(\mathbb{R}_{+}^{n+1})=\left\{v(r,x)\in D'\left(\mathbb{R}_{+}^{n+1}\right): \partial_r^{\alpha}\partial_x^{\beta}v(r,x)\in L^2\left(\mathbb{R}^{n+1},drdx\right)\right\},
 \end{equation}
for all  $\alpha\in \mathbb{N}, \beta\in \mathbb{N}^n$  and $|\alpha|+|\beta|\leq m$. The readers can be referred  more details and information on Fuchs
type operators and the weighted Mellin Sobolev spaces in \cite{es,sb}.

The space  $\mathcal{H}_2^{m,\gamma}(\mathbb{R}_{+}^{n+1})$ can be extended to more general cases  $ \mathcal{H}_p^{m,\gamma}(\mathbb{R}_{+}^{n+1})$
for $1\leq p<+\infty$ and $\mathcal{H}_p^{m,\gamma}(\mathbb{B})$ (the cone Sobolev spaces on manifolds with conical singularities).

\begin{definition}\label{d4}
 For $(t,x)\in  \mathbb{R}_{+}\times \mathbb{R}^n$, we say that $u(t,x)\in L_p\left(\mathbb{R}_{+}^{n+1},\frac{dt}{t}dx\right)$ if
  $$|u|_{L_p}=\left[\int_{\mathbb{R}_{+}\times \mathbb{R}^n}t^{n+1}|u(t,x)|^p \frac{dt}{t}dx\right]^{\frac{1}{p}}<+\infty.$$
Furthermore, the weighted $L_p$-spaces with weight data $\gamma\in \mathbb{R}$ is denoted by $L^{\gamma}_p\left(\mathbb{R}_{+}^{n+1},\frac{dt}{t}dx\right)$,
that is, if $u(t,x)\in L^{\gamma}_p\left(\mathbb{R}_{+}^{n+1},\frac{dt}{t}dx\right)$, then $t^{-\gamma}u(t,x)\in L_p\left(\mathbb{R}_{+}^{n+1},\frac{dt}{t}dx\right)$, and
  $$|u|_{L^{\gamma}_p}=\left[\int_{\mathbb{R}_{+}\times \mathbb{R}^n}t^{n+1}|t^{-\gamma}u(t,x)|^p \frac{dt}{t}dx\right]^{\frac{1}{p}}<+\infty.$$
\end{definition}

The weighted Sobolev space for all $1\leq p < +\infty$ can be defined as

\begin{definition}\label{d5}
 For $m\in \mathbb{N}$, the spaces
  \begin{equation}\label{e07}
   \mathcal{H}_p^{m,\gamma}(\mathbb{R}_{+}^{n+1})=\left\{u\in D'\left(\mathbb{R}_{+}^{n+1}\right): t^{\frac{n+1}{p}-\gamma}(t\partial_t)^{\alpha}\partial_x^{\beta}u\in L_p\left(\mathbb{R}_{+}^{n+1},\frac{dt}{t}dx\right)\right\},
 \end{equation}
 for all  $\alpha\in \mathbb{N}, \beta\in \mathbb{N}^n$  and $|\alpha|+|\beta|\leq m$.  In other words, if $u(t,x)\in \mathcal{H}_p^{m,\gamma}(\mathbb{R}_{+}^{n+1})$, then $(t\partial_t)^{\alpha}\partial_x^{\beta}u\in L_p^{\gamma}\left(\mathbb{R}_{+}^{n+1},\frac{dt}{t}dx\right)$.
\end{definition}

It is easy to see that  $\mathcal{H}_p^{m,\gamma}(\mathbb{R}_{+}^{n+1})$ is a Banach space with norm
$$\|u\|_{\mathcal{H}_p^{m,\gamma}(\mathbb{R}_{+}^{n+1})}=\sum_{|\alpha|+|\beta|\leq m}\left[\int_{\mathbb{R}_{+}\times \mathbb{R}^n}
t^{n+1} | t^{-\gamma}(t\partial_t)^{\alpha}\partial_x^{\beta}u(t,x)|^p \frac{dt}{t}dx\right]^{\frac{1}{p}}.$$

Similarly (see \cite{css}),  the weighted Sobolev spaces  $\mathcal{H}_p^{m,\gamma}(X^{\wedge})$ with $1\leq p<\infty$
 can be defined on manifolds with conical singularities. Let $X$ be a closed compact $C^{\infty}$ manifold, and $U =
\{U_1,\cdots,U_N \}$ an open covering of $X$ by coordinate neighborhoods. If we fix a subordinate
partition of unity $\{\phi_1,\cdots, \phi_N \}$ and charts $\chi_j: U_j\rightarrow \mathbb{R}^n, j=1,2,\cdots,N$, then $u\in \mathcal{H}_p^{m,\gamma}(X^{\wedge})$ if and only if $u\in D'(X^{\wedge})$ with the norm
$$\|u\|_{\mathcal{H}_p^{m,\gamma}(X^{\wedge})}=\left[\sum_{j=1}^N\left\|(1\times
\chi_j^{\ast})^{-1}\phi_ju\right\|^p_{\mathcal{H}_p^{m,\gamma}(\mathbb{R}_{+}^{n+1})}\right]^{\frac{1}{p}}<+\infty,$$
where $1\times\chi_j^{\ast}: C_0^{\infty}(\mathbb{R}_{+}\times \mathbb{R}^n)\rightarrow C_0^{\infty}(\mathbb{R}_{+}\times U_j)$
is the pull-back function with respect to
$1\times\chi_j: \mathbb{R}_{+}\times U_j\rightarrow \mathbb{R}_{+}\times \mathbb{R}^n$.
We denote  the closure of $C_0^{\infty}(X^{\wedge})$ with respect to the norm  $\|\cdot\|_{\mathcal{H}_p^{m,\gamma}(X^{\wedge})}$
 by $\mathcal{H}_{p,0}^{m,\gamma}(X^{\wedge})$.

\begin{lemma}[See\cite{ss}]\label{l2}
 For all $m\in \mathbb{N}$, $\gamma\in \mathbb{R}$, we have $\mathcal{H}_p^{m,\gamma}(X^{\wedge})\subset W_{\text{loc}}^{m,p}(X^{\wedge})$,
 where  $W_{\text{loc}}^{m,p}(X^{\wedge})$ denotes the subspace of all $u\in D'(X^{\wedge})$ such that $\phi u\in W^{m,p}(X^{\wedge})$
 for each $\phi \in C_0^{\infty}(X^{\wedge})$.
\end{lemma}

Let $\mathbb{B}$ be the stretched manifold of $B$, we will always denote $\omega(t) \in  C^{\infty}(\mathbb{B})$
 as a real-valued cut-off function which equals 1 near $\{0\} \times \partial B$.

\begin{definition}\label{d6}
Let $\mathbb{B}$ be the stretched manifold to a manifold B with conical singularities.
Then $\mathcal{H}_p^{m,\gamma}(\mathbb{B})$ for $m \in N, \gamma \in  \mathbb{R}$ denotes the subspace of all $u \in W^{m,p}_{\emph{loc}}(\mbox{int} (\mathbb{B}))$
 such that
  $$\mathcal{H}_p^{m,\gamma}(\mathbb{B})=\{u\in W^{m,p}_{\text{loc}}(\mbox{int} (\mathbb{B})): \omega u\in \mathcal{H}_p^{m,\gamma}(X^{\wedge}) \}$$
for any cut-off function $\omega$, supported by a collar neighbourhood of $[0, 1) \times \partial \mathbb{B}$, where $\mbox{int} (\mathbb{B})=\mathbb{B}\setminus \partial \mathbb{B}$. Moreover,
the subspace $\mathcal{H}_p^{m,\gamma}(\mathbb{B})$ is defined as follows:
 $$\mathcal{H}_p^{m,\gamma}(\mathbb{B})=[\omega] \mathcal{H}_{p,0}^{m,\gamma}(X^{\wedge}) +[1-\omega] W^{m,p}_{0}(\mbox{int} (\mathbb{B}))$$
where $W^{m,p}_{0}(\mbox{int} (\mathbb{B}))$ denotes the closure of $C^{\infty}_0 (\mbox{int} (\mathbb{B}))$ in Sobolev spaces $W^{m,p}(\tilde{X})$ when $\tilde{X}$ is
a closed compact $C^{\infty}$ manifold of dimension n + 1 that contains $\mathbb{B}$ as a submanifold with
boundary.
\end{definition}

\begin{lemma}[See\cite{ss}]\label{l3}
We have the following properties:
\begin{itemize} \addtolength{\itemsep}{-1.5 em} \setlength{\itemsep}{-5pt}
   \item [\emph{(1)}]  $\mathcal{H}_p^{m,\gamma}(\mathbb{B})$ is Banach space for $1\leq p<+\infty$, and is Hilbert space for $p = 2$.
    \item [\emph{(2)}] $L^{\gamma}_p(\mathbb{B})=\mathcal{H}_p^{0,\gamma}(\mathbb{B})$.
    \item [\emph{(3)}] $L^{\gamma}_p(\mathbb{B})=\mathcal{H}_p^{0,0}(\mathbb{B})$.
     \item [\emph{(4)}] $t^{\gamma_1}\mathcal{H}_p^{m,\gamma_2}(\mathbb{B})=\mathcal{H}_p^{m,\gamma_1+\gamma_2}(\mathbb{B})$.
     \item [\emph{(5)}]  The embedding  $\mathcal{H}_p^{m,\gamma}(\mathbb{B})\hookrightarrow \mathcal{H}_p^{m',\gamma'}(\mathbb{B})$ is continuous if
$m\geq m'$, $\gamma\geq \gamma'$ and is compact if $m> m'$, $\gamma> \gamma'$.
\end{itemize}
\end{lemma}

\section{Cone Moser-Trudinger  inequalities}
Let $P : \mathbb{R}_+ \rightarrow \mathbb{R}_+$ be an $N$-function, that is, $P$ is continuous, convex, with
$P(t)>0$ for $t>0$, $\frac{P(t)}{t}\rightarrow 0$ as $t\rightarrow0$, and $\frac{P(t)}{t}\rightarrow \infty$ as $t\rightarrow\infty$.
Equivalently, $P$ admits the representation $P(s)=\int_0^s p(\tau)d\tau$
 where $p:\mathbb{R}_+ \rightarrow \mathbb{R}_+$ is non-decreasing, right-continuous, with $p(0)=0, p(t)>0$ for $t>0$, and
$p(t)\rightarrow \infty$  as $\rightarrow \infty$.

The $N$-function $\tilde{P}$ conjugate to $P$ is defined by $\tilde{P}(t)=\int_0^t \tilde{p}(\tau)d\tau$, where
$\tilde{p}:\mathbb{R}_+ \rightarrow \mathbb{R}_+$   is given by $\tilde{p}(t)= \sup\{s :p(s)\leq t\}$ (see \cite{AF}).
It is easy to see that these $N$-functions can be extended into even functions on all $\mathbb{R}$.

The $N$-function $P$ is said to satisfy the $\Delta_2$ condition if, for some $k>0$,
\begin{equation}\label{e08}
  P(2t) \leq  kP(t), \, \,\,\forall t>0.
\end{equation}
When (\ref{e08}) holds only for $t$ at least some $t_0 >0$, then $P$ is said to satisfy the $\Delta_2$
condition near infinity.

Furthermore, the weighted Orlicz spaces with weight data $\gamma\in \mathbb{R}$ is denoted by $L^{\gamma}_P\left(\mathbb{R}_{+}^{n+1},\frac{dt}{t}dx\right)$,
that is, if $u(t,x)\in L^{\gamma}_P\left(\mathbb{R}_{+}^{n+1},\frac{dt}{t}dx\right)$, then $t^{-\gamma}u(t,x)\in L_p\left(\mathbb{R}_{+}^{n+1},\frac{dt}{t}dx\right)$, and
  $$\left[\int_{\mathbb{R}_{+}}\int_{ \mathbb{R}^n}t^{n+1}P\left(|t^{-\gamma}u(t,x)|\right) \frac{dt}{t}dx\right]<+\infty.$$
We easily know that $L^{\gamma}_P\left(\mathbb{R}_{+}^{n+1},\frac{dt}{t}dx\right)$ is a Banach space under the norm

$$|u|_{L^{\gamma}_P}=\inf\left\{\lambda>0: \int_{\mathbb{R}_{+}}\int_{ \mathbb{R}^n}t^{n+1}P\left(\frac{|t^{-\gamma}u(t,x)|}{\lambda}\right) \frac{dt}{t}dx<+\infty\right\}.$$
 $\mathcal{L}^{\gamma}_P\left(\mathbb{B},\frac{dt}{t}dx\right)$ is a convex subset of $L^{\gamma}_P\left(\mathbb{B},\frac{dt}{t}dx\right)$.

The closure in $L^{\gamma}_P\left(\mathbb{B},\frac{dt}{t}dx\right)$ of the set of bounded measurable functions with compact
support in $\mathbb{B}$ is denoted by $E_P (\mathbb{B})$. The equality $E_P (\mathbb{B})=L^{\gamma}_P\left(\mathbb{B},\frac{dt}{t}dx\right)$ holds if and only
if $P(s)$ satisfies the $\Delta_2$ condition, for all $s$ or for $s$ large according to whether $\mathbb{B}$ has
infinite measure or not. The dual of $E_P (\mathbb{B})$ can be identified with $L_{\tilde{P}}\left(\mathbb{B},\frac{dt}{t}dx\right)$ by means of the pairing
$\int_{\mathbb{B}} u(x)v(x) \frac{dt}{t}dx$, and the dual norm on $L_{\tilde{P}}\left(\mathbb{B},\frac{dt}{t}dx\right)$  is equivalent to $|\cdot|_{L^{\gamma}_{\tilde{P}}}$.

The space $L_{\tilde{P}}\left(\mathbb{B},\frac{dt}{t}dx\right)$ is reflexive if and only if $P$ and $\tilde{P}$ satisfy the $\Delta_2$ condition
(near infinity only if $\mathbb{B}$ has finite measure). $P\ll M$ means that $P$ grows essentially less rapidly than $M$, that is, for each $\varepsilon >0$,
$P(t)/(M(\varepsilon t))\rightarrow 0$ as $t\rightarrow \infty$. This is the case if and only if $M^{-1}(t)/P^{-1}(t)\rightarrow 0$ as
$t\rightarrow \infty$. Therefore, we have the continuous imbedding $L_M (B)\subset E_P (B)$
when $\mathbb{B}$ has finite measure.

Let $B$ be a n-dimensional compact manifold with conical singularity at the point $b \in \partial B$, and
$\mathbb{B}$ be the stretched manifold of $B$, i.e. without loss of generality, we suppose $\mathbb{B}=[0, 1)\times (-1,1)$, $X$
is a closed compact manifold of dimension $n - 1$, $\partial \mathbb{B} = {0} \times X$.

Next we denote $|\cdot|_{L_p^1}$,  $\|\cdot\|_{\mathcal{H}_p^{m,1}(\mathbb{B})}$
 by  $|\cdot|_{p}$,  $\|\cdot\|_{m,p}$, respectively.

\begin{theorem}\label{t0}
  Let $\mathbb{B}\in \mathbb{R}_{+}^{2}$, $u\in \mathcal{H}_2^{1,1}(\mathbb{B})$ and $|\nabla_{\mathbb{B}} u|_2^2\leq 1$.
Then there exists a constant $C>0$  such that
$$\int_{\mathbb{B}} e^{\alpha u^2}\frac{dx_1}{x_1}dx_2 \leq C (\mathbb{B}),$$
where $\alpha \leq \alpha_2=2\omega_1=4\pi$, $\nabla_{\mathbb{B}}=(x_1\partial_{x_1},\partial_{x_2})$,  and
$\omega_1=2\pi$ is  the perimeter of the unit sphere.

\end{theorem}
\begin{proof}
Let $B_r(1,0)$ be a  ball in $ \mathbb{R}_{+}^{2}$ with radius $R$, that is,
 $B_R(1,0)=\{(x_1, x_2)\in \mathbb{R}_{+}^{2}: |x|^2=|\ln x_1|^2+x_2^2\leq R^2 \}$.

We may assume that $u \geq 0$ since
we can replace $u$ by $|u|$ without increasing the integral of the gradient. Also, it
suffices to prove the statement for a set of functions $u$ which is dense in the unit
ball  of $\mathcal{H}_2^{1,1}(\mathbb{B})$. For example, we may assume that $u$ has compact support and
is in $C_0^{\infty}(\mathbb{B})$, or instead of the last requirement, is piecewise linear.

We use symmetrization: with $u(x)> 0$ we associate a function $u^{\ast}(x)$
depending on $|x|$ only by the requirement
$$|\{x: u^{\ast} > \rho\}| = |\{x \in  \mathbb{B} : u > \rho \mbox{\ for\  each\  } \rho\geq 0 \}|.$$

Clearly, $u^{\ast}$ is a decreasing function of $|x|$ which is 0 for $|x|>R$ where $R$ is the
radius of the sphere whose volume is

$$|(B_R(0,1))| = \int_{|x|\leq R}\frac{dx_1}{x_1}dx_2.$$
We define $\mathbb{B}^{\ast}$ as $|x|\leq R$. Similar to the Laplace operator,  we can build   heat kernel  theory  for the operator $\Delta_{\mathbb{B}}$. Hence we easily obtain  the following P\'{o}lya-Szeg\"{o} inequality
$$|\nabla_{\mathbb{B}} v^{\ast}|_2^2\leq |\nabla_{\mathbb{B}} v|_2^2,$$
while,
$$\int_{\mathbb{B}^{\ast}}e^{\alpha (u^{\ast})^p} \frac{dx_1}{x_1}dx_2=\int_{\mathbb{B}}e^{\alpha u^p} \frac{dx_1}{x_1}dx_2.$$

This reduces the problem at once to a one dimensional one. For convenience
we introduce the variable $t$ by $$\frac{|x|^2}{R^2}=e^{-t},$$
and set
$$w(t)=(2\omega_1)^{\frac{1}{2}}u^{\ast}(x).$$

Since spherical coordinate integral formula still holds in this sense,  $w$ is monotone increasing and
$$\int_0^{\infty} \dot{w}^2dt=\int_{\mathbb{B}^{\ast}} |\nabla_{\mathbb{B}} u^{\ast}|^2\frac{dx_1}{x_1}dx_2,$$
$$\int_0^{\infty} e^{\beta w^p-t} dt=\frac{1}{|\mathbb{B}^{\ast}|}\int_{\mathbb{B}^{\ast}} e^{\alpha (u^{\ast})^p}\frac{dx_1}{x_1}dx_2,$$
where $\beta=\alpha/\alpha_2$. Thus it is sufficient to prove:

If $q \geq 2$ and $w(t)$ is a $C^1$-function and $0 \leq t < \infty$ satisfying
\begin{equation}\label{e36}
 w(0)=0, \dot{w}\geq 0, \int_0^{\infty} \dot{w}^q(t)dt\leq 1,
\end{equation}
then
\begin{equation}\label{e37}
 \int_0^{\infty} e^{\beta w^p-t} dt\leq C_1,\ \ \mbox{if}\ \ \beta\leq 1, \frac{1}{p}+\frac{1}{q}=1,
\end{equation}
where the constant $C_1$  depends on $q$ only.

From H\"{o}lder's inequality
$$w(t)=\int_0^t \dot{w}(t)dt\leq t^{\frac{1}{p}}\left(\int_0^t \dot{w}^q(t)dt\right)^{\frac{1}{q}}\leq t^{\frac{1}{p}},$$
it is clear that
\begin{equation}\label{e38}
 \int_0^{\infty} e^{\beta w^p-t} dt\leq \int_0^{\infty} e^{(\beta-1)t}=\frac{1}{1-\beta},\ \ \mbox{for}\ \ \beta< 1.
\end{equation}

But for $\beta = 1$  we have to proceed more carefully. The same simple device allows
to show that the integral in (\ref{e37}) exists for any positive $\beta$. Indeed, given any $\varepsilon > 0$
there exists  $T = T(\varepsilon)$ such that $$\int_T^{\infty} \dot{w}^qdt<\varepsilon,$$
from which we conclude, again by H\"{o}lder inequality, that
$$w(t)\leq w(T)+\varepsilon^{\frac{1}{q}}(t-T)^{\frac{1}{p}},\ \ \mbox{for}\ \ t\geq T,$$
hence
$$\lim_{t\rightarrow\infty} \frac{w(t)}{t^{\frac{1}{p}}}=0.$$
Thus $\beta w^p < \frac{1}{2}t$ for sufficiently large $t$, which makes the existence of the integral
in (\ref{e37}).

Next we show that for $\beta > 1$,  this integral can be made arbitrarily large.
For this purpose we let $\eta(s) = \min \{s, 1\}$ and set  $w= t_1^{\frac{1}{p}}\eta(t_j/t_1)$. Then clearly,
this function satisfies (\ref{e36}) but
$$\int_0^{\infty} e^{\beta w^p-t} dt\geq \int_{t_1}^{\infty} e^{\beta t_1-t} dt=e^{(\beta-1)t_1}$$
tends to infinity as $t_1 \rightarrow \infty$.
\end{proof}

\begin{theorem}\label{t2}
For $\alpha\in (0,\alpha_2)$, there exists a constant $C=C(\alpha)>0$ such that
\begin{equation}\label{e19}
\int_{\mathbb{R}^{2}_{+}}\left( e^{\alpha\frac{|u(x)|^2}{|\nabla_{\mathbb{B}}  u(x)|_2^2}}-1\right)\frac{dx_1}{x_1}dx_2\leq C_{\alpha}\frac{|u(x)|_2^2}{|\nabla_{\mathbb{B}}  u(x)|_2^2}
\end{equation}
for $u\in \mathcal{H}_2^{1,1}(\mathbb{R}^2_{+})\setminus\{0\}$ . In particular, if $A_{\alpha}(s)=e^{\alpha s^2}-1$, $s\geq 0$, then
\begin{equation}\label{e35}
 |u(x)|_{A}\leq |\nabla_{\mathbb{B}} u(x)|_2
\end{equation}
for all $u\in \mathcal{H}_2^{1,1}(\mathbb{R}^2_{+})$.
\end{theorem}
\noindent \textbf{Remark}: Note that the inequality  (\ref{e19}) is scale invariant, that is, if for $r>0$ we define the rescaled  function
$$u_{r}(x_1,x_2)=u(x_1^r,rx_2),\ \ \ (x_1,x_2)\in \mathbb{R}_{+}^{2},$$
then
$$\int_{\mathbb{R}^2_{+}}u^2(x^r_1,rx_2)\frac{dx_1}{x_1}dx_2=\frac{1}{r^2}\int_{\mathbb{R}^2_{+}} u^2(y_1,y_2)\frac{dy_1}{y_1}dy_2,$$
\begin{align*}
 \int_{\mathbb{R}^2_{+}}|\nabla_{\mathbb{B}} u_r|^2\frac{dx_1}{x_1}dx_2=\int_{\mathbb{R}^2_{+}}|\nabla_{\mathbb{B}} u|^2\frac{dx_1}{x_1}dx_2.
\end{align*}
and
$$\int_{\mathbb{R}^{2}_{+}}\left( e^{\alpha\frac{|u(x)|^2}{|\nabla_{\mathbb{B}}  u(x)|_2^2}}-1\right)\frac{dx_1}{x_1}dx_2=\frac{1}{r^2}\int_{\mathbb{R}^{2}_{+}}\left( e^{\alpha\frac{|u_r(x)|^2}{|\nabla_{\mathbb{B}}  u_r(x)|_2^2}}-1\right)\frac{dx_1}{x_1}dx_2.$$

\begin{proof}
Fix $u\in \mathcal{H}_2^{1,1}(\mathbb{R}^2_{+})\setminus\{0\}$ and define
$$v(x)=\frac{|u(x)|}{|\nabla_{\mathbb{B}} u (x)|_2}.$$
Then $\|\nabla_{\mathbb{B}} v\|_2=1$ and (\ref{e19}) reduces to
\begin{equation}\label{e33}
 \int_{\mathbb{R}^2_{+}}e^{\alpha v^2(x)}\frac{dx_1}{x_1}dx_2\leq C_{\alpha}|v(x)|^2_{2}.
\end{equation}
Let $v^{\ast}$  be the spherically symmetric rearrangement of $v$.
Then $v^{\ast}(x)=w(|x|)$, where $w$ is nonnegative, decreasing, and locally absolutely continuous. Hence,
$$\nabla_{\mathbb{B}}v^{\ast}=w'(|x|)\frac{x}{|x|}.$$
Similar to the Laplace operator,  we can build   heat kernel  theory  for the operator $\Delta_{\mathbb{B}}$.
Hence we easily obtain  the following P\'{o}lya-Szeg\"{o} inequality
$$|\nabla_{\mathbb{B}} v^{\ast}|_2^2\leq |\nabla_{\mathbb{B}} v|_2^2.$$
Now using spherical coordinates, if follows that
\begin{align}\label{e23}
  &\omega_1\int_0^{\infty} |w'(r)|^2rdr=|\nabla_{\mathbb{B}} v^{\ast}|_2^2\leq |\nabla_{\mathbb{B}} v|_2^2=1.
  \end{align}
  \begin{align}\label{e32}
  &\int_{\mathbb{R}^{2}_{+}} (v^{\ast})^2\frac{dx_1}{x_1}dx_2
= \int_{\mathbb{R}^{2}_{+}} v^2\frac{dx_1}{x_1}dx_2.
\end{align}
Define
\begin{equation}\label{e22}
r_0:=\inf\{r\geq 0: w(r)\leq 1\}.
\end{equation}
Since $w(r)\rightarrow 0$ as $r\rightarrow \infty$, we have that $r_0$ must be finite.

Using spherical coordinates, we have that
\begin{align}\label{e21}
\int_{\mathbb{R}^2_{+}} e^{\alpha v^2} \frac{dx_1}{x_1}dx_2&= \int_{\mathbb{R}^2_{+}} e^{\alpha (v^{\ast})^2} \frac{dx_1}{x_1}dx_2\nonumber\\
&= \omega_1\int_0^{\infty} e^{\alpha w^2(r)} rdr\nonumber\\
&=\omega_1\int_0^{r_0} e^{\alpha w^2(r)} rdr+\omega_1\int_{r_0}^{\infty} e^{\alpha w^2(r)} rdr=: I+II.
\end{align}

To estimate $I$, it is enough to consider the case that $r_0 > 0$, so that $w (r_0) =1 $ by (\ref{e22}).
 Since $w$ is locally absolutely continuous, by the fundamental
theorem of calculus, H\"{o}lder's inequality, (\ref{e23}) and
(\ref{e22}), for $0 < r < r_0$, we have that
\begin{align}\label{e24}
 w(r)&=w(r_0)-\int_r^{r_0} w'(\tau) d\tau\leq 1+\int_{r}^{r_0} |w'(\tau)|\frac{\tau^{2}}{\tau^{2}}d\tau\nonumber\\
 &\leq 1+\left(\int_0^{\infty} |w'(\tau)|^2\tau dr\right)^2\left(\ln\left(\frac{r_0}{r}\right)\right)^{\frac{1}{2}}\nonumber\\
 &\leq 1+\omega_1^{-\frac{1}{2}} \left(\ln\left(\frac{r_0}{r}\right)\right)^{\frac{1}{2}}.
\end{align}
By the convexity of the function $s^2$, for every $\varepsilon > 0$  we may find a constant
$C_{\varepsilon} > 0$ such that
$$(1+s)^2\leq (1+\varepsilon)s^2+C_{\varepsilon}$$
for all $s\geq 0$. Hence,
\begin{equation}\label{e25}
  w^2(r)\leq (1+\varepsilon)\omega_1^{-1} \ln\left(\frac{r_0}{r}\right)+C_{\varepsilon}
\end{equation}
for all $0<r<r_0$. Since $0<\alpha<\alpha_2$, we may take $\varepsilon>0$ so small that
$$\alpha(1+\varepsilon)<\alpha_2=2\omega_1,$$
and so $\alpha(1+\varepsilon)\beta^{-1}<2$. Hence, by \eqref{e25},

\begin{align}\label{e26}
I&\leq \omega_1 e^{\alpha C_{\varepsilon}}\int_{0}^{r_0} e^{\alpha (1+\varepsilon)\omega_1^{-1} \ln\left(\frac{r_0}{r}\right)}rdr\nonumber\\
&=\omega_1 e^{\alpha C_{\varepsilon}}r_0^{\alpha(1+\varepsilon)\beta^{-1}}\int_0^{r_0} r^{1-\gamma(1+\varepsilon)\omega_1^{-1}}dr\nonumber\\
&=\frac{\omega_1 e^{\alpha C_{\varepsilon}}}{1-\gamma(1+\varepsilon)\omega_1^{-1}} r_0^2=: C_1(2,\alpha)r_0^2.
\end{align}
On the other hand, by the Lebesgue monotone convergence theorem and the
fact that $w(r)<1$  for all $r > r_0$,  by (\ref{e22}), we have that

\begin{align}\label{e29}
II&\leq \omega_1 \sum_{n=1}^{\infty} \frac{1}{n!}\alpha^n\int_{r_0}^{\infty} w^{2n}(r)rdr\nonumber\\
&\leq \omega_1 \sum_{n=1}^{\infty} \frac{1}{n!}\alpha^n\int_{r_0}^{\infty} w^{2}(r)rdr\nonumber\\
&=\omega_1 e^{\alpha}\int_{r_0}^{\infty} w^2(r)rdr.
\end{align}
Combining this estimate with (\ref{e21}) and (\ref{e26}), we get
\begin{align}\label{e30}
\int_{\mathbb{R}^2_{+}} e^{\alpha v^2} \frac{dx_1}{x_1}dx_2\leq C_1(2,\alpha)r_0^2+\omega_1 e^{\alpha}\int_{r_0}^{\infty} w^2(r)rdr.
\end{align}
Using spherical coordinates and (\ref{e32}), one has that
\begin{equation}\label{e31}
 e^{\alpha}\omega_1\int_{r_0}^{\infty} w^2(r)rdr\leq \int_{\mathbb{R}^{2}_{+}} (v^{\ast})^2\frac{dx_1}{x_1}dx_2
= \int_{\mathbb{R}^{2}_{+}} v^2\frac{dx_1}{x_1}dx_2.
\end{equation}
Thus, to obtain (\ref{e33}), it remains to estimate in the case that $r_0 > 0$.
By (\ref{e22}) and the fact that w is decreasing, we have that $w(r)>w(r_0)=1$
if and only if $0<r<r_0$. Hence,
$$\{x\in  \mathbb{R}^{2}_{+}: v^{\ast}(x)>1\}=\{x\in  \mathbb{R}^{2}_{+}: w(|x|)>1\}=B_{r_0}(0,1).$$
Furthermore, we have
\begin{align}\label{e34}
\alpha_2r_0^2&= \big|\{x\in  \mathbb{R}^{2}_{+}: v^{\ast}(x)>1\}\big|\nonumber \\
&=\big|\{x\in  \mathbb{R}^{2}_{+}: v(|x|)>1\}\big|\nonumber \\
&\leq \int_{\{v>1\}} v^2 \frac{dx_1}{x_1}dx_2,
\end{align}
and so we have proved (\ref{e33}), and in turn, (\ref{e19}).

By (\ref{e19}), the number $s= |\nabla_{\mathbb{B}} u|_2^2 $ is admissible in the definition
of $|u|_A$,  and so (\ref{e35}) follows.
\end{proof}

\begin{theorem}\label{t3}
For $\alpha\geq \alpha_2$, there exists a sequence $\{u_k(x)\}\subset \mathcal{H}_2^{1,1}(\mathbb{R}^2_{+})$ such that
$|\nabla u|_2=1$ and
\begin{equation}\label{e50}
\frac{1}{|u_k|_2^2}\int_{\mathbb{R}^{2}_{+}} \left(e^{\alpha\frac{|u_k(x)|^2}{|\nabla_{\mathbb{B}}  u_k(x)|_2^2}}-1\right)\frac{dx_1}{x_1}dx_2\geq \frac{1}{|u_k|_2^2}\int_{\mathbb{R}^{2}_{+}} \left(e^{\alpha_2\frac{|u_k(x)|^2}{|\nabla_{\mathbb{B}}  u_k(x)|_2^2}}-1\right)\frac{dx_1}{x_1}dx_2\rightarrow \infty,
\end{equation}
as $k\rightarrow\infty$.
\end{theorem}

\begin{proof}
We shall take  similar  argument in the proof of Theorem \ref{t2}.  It suffices to find a sequence of
functions $w(r)=u_k(|x|)$ which satisfies
\begin{align}\label{e52}
  &\omega_1 \int_0^{\infty} |w_k'(r)|^2rdr= |\nabla_{\mathbb{B}} u_k|_2^2=1,
  \end{align}
and
\begin{align}\label{e53}
  &\int_{\mathbb{R}^{2}_{+}} u_k^2(|x|)\frac{dx_1}{x_1}dx_2
= \omega_1\int_0^{\infty} w_k^{2}(r)rdr\rightarrow 0,
\end{align}

\begin{equation}\label{e54}
\int_{\mathbb{R}^2_{+}} e^{\alpha u_k^2(|x|)} \frac{dx_1}{x_1}dx_2=\omega_1\int_0^{\infty} e^{\alpha w_k^2(r)}rdr\geq \frac{1}{2}.
\end{equation}
Here we give an example of $u_{k}(r)$ explicitly. We set
\begin{equation}
 u_k(r)=\left \{
\begin{aligned}
  &0, &\mbox{if}\ r\geq 0, \\
   &\frac{-2\ln r}{\sqrt{2\omega_1}}k^{-\frac{1}{2}},\ \quad &\mbox{if}\   e^{-\frac{k}{2}}< r\leq 1, \\
  & \frac{1}{\sqrt{2\omega_1}} k^{\frac{1}{2}},\  &\mbox{if}\   0< r\leq e^{-\frac{k}{2}}.
   \end{aligned}\right.
\end{equation}
It is easily see  that $u_{k}$ satisfies (\ref{e52})--(\ref{e54}).

\end{proof}


\section{Nonlinear Dirichlet boundary value problems}
In this section, we consider the following Dirichlet boundary value problems
\begin{equation}
\left\{
\begin{array}{ll}
-\Delta_{\mathbb{B}} u=f(x,  u), &\mbox{in}\  x\in \mbox{int} (\mathbb{B}),    \\
u= 0,  &\mbox{on}\  \partial\mathbb{B},
\end{array}
\right.    \label{P}
\end{equation}
where, $-\Delta_{\mathbb{B}} =(x_1\partial_{x_1})^2+(\partial_{x_2})^2$, $f$ is  a continuous real function and satisfies the following assumptions:
\begin{itemize}\addtolength{\itemsep}{-1.5 em} \setlength{\itemsep}{-5pt}

\item[$(f_1)$]  $f\in C(\bar{\mathbb{B}}\times \mathbb{R})$ with $f(x, 0)=0$ and  $f(x,t)$ has subcritical exponential growth on $\mathbb{B}$, i,e,
  $$\lim_{t\rightarrow +\infty }\frac{|f(x,t)|}{e^{\alpha t^2}}=0,  \mbox{\ uniformly on\ }   x\in \mathbb{B}  \mbox{\ for all} \ \alpha>0, $$

\item[$(f'_{1})$]   $f\in C(\bar{\mathbb{B}}\times \mathbb{R})$ with $f(x, 0)=0$ and $f(x,t)$ critical exponential growth on $\mathbb{B}$, i,e, there exists $\alpha_0>0$ such that
   $$\lim_{t\rightarrow +\infty }\frac{|f(x,t)|}{e^{\alpha t^2}}=0,  \mbox{\ uniformly on\ }   x\in \mathbb{B}  \mbox{\ for all}\  \alpha>\alpha_0, $$

and

 $$\lim_{t\rightarrow +\infty }\frac{|f(x,t)|}{e^{\alpha t^2}}=+\infty,  \mbox{\ uniformly on\ }   x\in \mathbb{B}  \mbox{\ for all}\  \alpha<\alpha_0, $$

    \item[($f_2$)]
  $$\displaystyle\lim_{|t|\rightarrow \infty}\frac{F(x, t)}{|t|^{2}}=+\infty, \mbox{\ uniformly on\ }   x\in \mathbb{B}.$$

  \item[$(f_{3})$]   there exists $\theta\geq 1$ such that $\theta\mathcal{F}(x, t)\geq \mathcal{F}(x, st)$ for $(x, t)\in \mathbb{B}\times\mathbb{R}$ and $s\in[0, 1]$,  where, $F (x, t):=\int_0^tf (x, s) ds,  \mathcal{F}(x, t):= f (x, t)t - 2F(x, t)$,


   \item[$(f_{4})$]    $$\limsup_{t\rightarrow 0^{+}}\frac{|2F(x,t)|}{ |t|^2}<\lambda_1,  \mbox{\ uniformly on\ }   x\in \mathbb{B}.$$
  where  $\lambda_1$ is the first eigenvalue of $-\Delta_{\mathbb{B}}$   with Dirichlet problem (see \cite{CLW2}).

\end{itemize}

We define the functional $$I(u)=\frac{1}{2}\int_{\mathbb{B}} |\nabla_{\mathbb{B}} u|^2\frac{dx_1}{x_1}dx_2-\int_{\mathbb{B}}F(x,u)\frac{dx_1}{x_1}dx_2,\ u\in \mathcal{H}_{2,0}^{1,1}(\mathbb{B}).$$
It is easy to check that $I\in C^1(\mathcal{H}_{2,0}^{1,1}(\mathbb{B}), \mathbb{R})$, and
 the critical point of $I$ are precisely the weak solutions of problem (\ref{P}). We will prove the
existence of such critical points by the Mountain Pass Theorem.  Recently, there are some
interesting results about nonlinear differential equations on manifolds with conical singularities (see \cite{Alim,CLiu,CLW, CLW2, Chen3}).

 \begin{definition}\label{d43}
 Let $(X,\|\cdot\|_{X})$ be a reflexive  Banach space with its dual space $(X^{\ast},\|\cdot\|_{X^{\ast}})$ and $I\in C^1(X,\mathbb{R})$. For $c\in\mathbb{R}$,
 we say that $I$ satisfies the $(C)_c$ condition if for any sequence $\{x_n\}\subset X$ with
 $$I(x_n)\rightarrow c,\ \ (1+\|x_n\|)\|I'(x_n)\|\rightarrow 0\ \ \mbox{in}\ X^{\ast},$$
 there is  a subsequence $\{x_{n_k}\}$ such that $\{x_{n_k}\}$ converges strongly in $X$.
\end{definition}
\begin{proposition}[See \cite{MR0370183}, Mountain Pass Theorem]\label{mmp1}
Let $(X,\|\cdot\|_{X})$ be a reflexive  Banach space, $I\in C^1(X,\mathbb{R})$ satisfies the $(C)_c$ condition for any $c\in \mathbb{R}$, $I(0)=0$ and
\begin{itemize}\addtolength{\itemsep}{-1.5 em} \setlength{\itemsep}{-5pt}
  \item[(1)] There are constants $\rho,\alpha>0$ such that $I\mid_{\partial B_{\rho}}\geq \alpha$;
  \item[(2)] There exists   $e\in X\setminus B_{\rho}$ such that $I(e)\leq 0$.
\end{itemize}
Then $\displaystyle c=\inf_{\gamma\in \Gamma}\max_{\ 0\leq t\leq 1} I(\gamma(t))\geq \alpha$ is a critical point of $I$ where
$$\Gamma=\{\gamma\in C^0([0,1],X), \gamma(0)=0, \gamma(1)=e\}.$$
\end{proposition}

Next we denote  $\|\cdot\|_{\mathcal{H}_{2,0}^{1,1}(\mathbb{B})}$
 by $\|\cdot\|$, and we can prove the following results:
\begin{theorem}\label{t41}
Assume that $(f’_{1})$--$(f_{4})$   are satisfied, then problem (\ref{P}) has a nontrivial solution in $\mathcal{H}_{2,0}^{1,1}(\mathbb{B})$.
\end{theorem}

\begin{theorem}\label{t42}
Assume that $(f'_{1})$, $(f_{2})$--$(f_{4})$   and
\begin{itemize}\addtolength{\itemsep}{-1.5 em} \setlength{\itemsep}{-5pt}
 \item[$(f_{5})$] $\displaystyle\lim_{t\rightarrow+\infty} f(x,t)t e^{-\alpha_0t^2}\geq \beta>\left(\frac{2}{d}\right)^2\frac{1}{M\alpha_0}$, uniformly in $(x,t)$
 where $d$ is the inner radius of $\mathbb{B}$, i.e. $d:=$ radius of the largest open ball $\subset \mathbb{B}$,
 $$M=\lim_{n\rightarrow\infty} n\int_0^1 e^{n(t^2-t)}dt\ (\geq 2),$$
 and
\item[$(f_{6})$] $f$ is class $(A_0)$, i.e. for any $\{u_n\}$ in $\mathcal{H}_{2,0}^{1,1}(\mathbb{B})$, if
$
\left\{
\begin{array}{ll}
u_n\rightharpoonup 0 &\mbox{in}\   \mathcal{H}_{2,0}^{1,1}(\mathbb{B}),    \\
f(x,u_n)\rightarrow 0,  &\mbox{in}\  L_1^1(\mathbb{B}),
\end{array}
\right.
$  then $F(x,u_n)\rightarrow 0$ in $L^1_1(\mathbb{B})$ (up to a subsequence),
\end{itemize}
 are satisfied, then problem (\ref{P}) has a nontrivial solution in $\mathcal{H}_{2,0}^{1,1}(\mathbb{B})$.
\end{theorem}

The following lemmas will be used for proving our problems.


\begin{lemma}\label{l46}
Let $f$ satisfy  $(f_{2})$. Then $I(tu)\rightarrow-\infty$ as $t\rightarrow \infty$ for all nonnegative function
$u\in\mathcal{H}_{2,0}^{1,1}(\mathbb{B})\setminus \{0\}$.
\end{lemma}
\begin{proof}
By the condition $(f_{2})$,  there exist constants $C_1,C_2$ such that   $$F(x,t)\geq C_1|t|^{\theta}-C_2.$$
Then
\begin{align}\label{e20}
  I(tu)&\leq \frac{t^2}{2}\|u\|^2-C_1t^{\theta}\int_{\mathbb{B}} |u|^{\theta}\frac{dx_1}{x_1}dx_2+C_2\nonumber\\
  &\leq \frac{t^2}{2}\left(\|u\|^2-C_1\int_{\mathbb{B}} |u|^{\theta}\frac{dx_1}{x_1}dx_2\right)+C_2.
\end{align}
Now, choose $M >\frac{\|u\|^2}{2|u|_2^2}$, we have $J(tu)\rightarrow \infty$ as $t\rightarrow \infty$, so $I$ satisfies (ii) of Proposition \ref{mmp1}.
\end{proof}

\begin{lemma}\label{l47}
Let $f$ satisfy  $(f_{1})$ and $(f_{4})$ . Then there exits $\delta,\rho>0$ such that
$$I(u)\geq \delta,\ \ \mbox{if}\ \ \|u\|=\rho.$$
\end{lemma}
\begin{proof}
Using $(f_{1})$ and $(f_{4})$, there exists $k,\tau>0$ and $q>2$ such that
$$F(x,s)\leq \frac{1}{2} (\lambda_1-\tau)|s|^2+C|s|^qe^{ks^2},\ \ \mbox{for\ all}\ (x,s)\in \mathbb{B}\times \mathbb{R}.$$
By H\"{o}lder's inequality and the cone Moser-Trudinger embedding, we have
\begin{align}\label{e55}
  \int_{\mathbb{B}} |u|^qe^{ku^2}\frac{dx_1}{x_1}dx_2 &
  \leq  \left( \int_{\mathbb{B}} e^{kr\|u\|^2\frac{u^2}{\|u\|^2}}\frac{dx_1}{x_1}dx_2\right)^{\frac{1}{r}}\cdot  \left(\int_{\mathbb{B}} |u|^{r'q}\frac{dx_1}{x_1}dx_2\right)^{\frac{1}{r'}}\nonumber\\
  &\leq   C\left(\int_{\mathbb{B}} |u|^{r'q}\frac{dx_1}{x_1}dx_2\right)^{\frac{1}{r'}},
\end{align}
if $r>1$ sufficiently close to 1 and $\|u\|\leq \sigma$, where $κr\sigma^2 < \alpha_2$. Thus by the definition of $\lambda_1$ and
the Sobolev embedding:
$$I(u)\geq \frac{1}{2}\left(1-\frac{(\lambda_1-\tau)}{\lambda}\right)\|u\|^2-C\|u\|^q.$$
Since $\tau > 0$ and $q > 2$, we may choose $\rho, \delta > 0$ such that $I(u)\geq \delta$ if $\|u\|=\rho$.
\end{proof}

\begin{lemma}\label{l48}
Let $f$ satisfy $(f_{1})$--$(f_{3})$. Then the functional satisfies $(C)_c$ condition for $c\in \mathbb{R}$.
\end{lemma}
\begin{proof}
Let $\{u_n\}$ be a  $(C)_{c}$ sequence of $I$.   We first show that  $\{u_n\}$ is bounded.
If $\{u_n\}$ is unbounded, up to a subsequence we may assume that for some $c\in \mathbb{R}$,
\begin{equation}\label{e054}
I(u_n)\rightarrow c,\ \ \|u_n\|\rightarrow \infty,\ \ \|I^{\prime}(u_n)\|\cdot\|u_n\|\rightarrow 0.
\end{equation}
So we have
\begin{equation}\label{e055}
\displaystyle\lim_{n\rightarrow\infty}\left(\int_{\Omega}\frac{1}{2}\int_{\Omega}\mathcal{F}(x, u_n)dx\right)=\displaystyle\lim_{n\rightarrow\infty}\left\{I(u_n)-\frac{1}{2}\langle I^{\prime}(u_n), u_n\rangle\right\}=c,
\end{equation}
Let $w_n=\frac{u_n}{\|u_n\|}$,  up to a subsequence we may assume that
\begin{equation}
w_n\rightharpoonup w \ \mathrm{in}\ \mathcal{H}_{2,0}^{1,1}(\mathbb{B}),\ \ w_n\rightarrow w\ \mathrm{in}\ L_p^1(\mathbb{B}),\ \ w_n\rightarrow w\ \mathrm{a.e.}\ x\in\mathbb{B}.
\end{equation}
We may similarly show that $w^{+}_n\rightarrow w^{+}$ in   $\mathcal{H}_{2,0}^{1,1}(\mathbb{B})$, where $w^{+}=\max\{w,0\}$.
If $w=0$, similar to $p$-Laplacian case in \cite{MR1718530, MR1828880},
 we can choose a sequence $\{t_n\}\subset \mathbb{R}$ such that
\begin{equation}
 I(t_nu_n)=\max_{t\in[0, 1]}I(tu_n).
\end{equation}
For any given $R>0$, by $(f_1)$ , there exists $C=C(m)>0$ such that
\begin{equation}\label{e62}
 F(x,s)\leq C|s|+e^{\frac{\alpha_2}{m^2}s^2}, \ \ \mbox{for\ all} \ (x,s)\in \mathbb{B}\times \mathbb{R}.
\end{equation}
Also since $\|u_n\|\rightarrow \infty$, we have
\begin{equation}\label{e64}
 I(t_nu_n)\geq I\left(\frac{m}{\|u_n\|}u_n\right)=I(mw_n)
\end{equation}
and by  \eqref{e62} and the fact $\int_{\mathbb{B}}F(x,w_n) \frac{dx_1}{x_1}dx_2=\int_{\mathbb{B}}F(x,w_n) \frac{dx_1}{x_1}dx_2$, we obtain
\begin{align}\label{e63}
  2I(mw_n) & \geq  m^2-2cm\int_{\mathbb{B}}|w_{n}^{+}| \frac{dx_1}{x_1}dx_2-2\int_{\mathbb{B}} e^{\alpha_2 |w_n^{+}|^2}\frac{dx_1}{x_1}dx_2\nonumber\\
   & \geq   m^2-2cm\int_{\mathbb{B}}|w_{n}^{+}| \frac{dx_1}{x_1}dx_2-2\int_{\mathbb{B}} e^{\alpha_2 |w_n|^2}\frac{dx_1}{x_1}dx_2.
\end{align}
Since $\|w_n\| = 1$, we have that
$\int_{\mathbb{B}} e^{\alpha_2 |w_n|^2}\frac{dx_1}{x_1}dx_2$ is bounded by a universal constant $C (\mathbb{B}) >0$ by the Moser-Trudinger inequality.
Also, since $w_n^{+}\rightharpoonup 0$ in  $\mathcal{H}_{2,0}^{1,1}(\mathbb{B})$,
we have that
$\int_{\mathbb{B}}|w_{n}^{+}| \frac{dx_1}{x_1}dx_2\rightarrow 0$. Thus using (\ref{e64}) and letting $n\rightarrow \infty$ in (\ref{e63}),
and then letting $m\rightarrow \infty$, we get
$$I(t_nu_n)\rightarrow\infty.$$

Note that  $I(0)=0$, $I(u_n)\rightarrow c$, we see that $t_n\in(0, 1)$ and
\begin{align}
\int_{\mathbb{B}}|\nabla t_nu_n|^{2)}\frac{dx_1}{x_1}dx_2&-\int_{\mathbb{B}}f(x, t_nu_n)t_nu_n\frac{dx_1}{x_1}dx_2=\langle I^{\prime}(t_nu_n), t_nu_n\rangle\nonumber\\
&=t_n\frac{d}{dt}\bigg|_{t=t_n}I(tu_n)=0.
\end{align}
 Therefore, by the condition ($f_{3}$),
 \begin{align}\label{e312}
\frac{1}{2}\int_{\mathbb{B}}\mathcal{F}(x, u_n)\frac{dx_1}{x_1}dx_2&\geq\frac{1}{2}\int_{\mathbb{B}}\frac{\mathcal{F}(x, t_nu_n)}{\theta}\frac{dx_1}{x_1}dx_2\nonumber\\
&\geq\frac{1}{\theta}\left(I(t_nu_n)-\frac{1}{2}\langle I'(t_nu_n), t_nu_n\rangle\right)\nonumber\\
&=\frac{1}{\theta}I(t_nu_n)\rightarrow \infty.
\end{align}
This contradicts with \eqref{e055}.

Now from the first limit in \eqref{e054}, when $\|u_n\|\geq 1$ we obtain
\begin{equation}\label{new}
\frac{1}{2}\|u_n\|^{2}-(c+o(1))\geq\int_{\mathbb{B}}F(x, u_n)\frac{dx_1}{x_1}dx_2.
\end{equation}
Using \eqref{new} and the condition ($f_2$) we deduce
\begin{align}\label{e514}
\frac{1}{2}-\frac{c+o(1)}{\|u_n\|^{2}}&\geq\int_{\mathbb{B}}\frac{F(x, u^{+}_n)}{\|u_n\|^{2}}\frac{dx_1}{x_1}dx_2\nonumber\\
&=\left(\int_{w=0}+\int_{w\not=0}\right)\frac{F(x, u^{+}_n)}{|u^{+}_n|^{2}}|w^{+}_n|^{2}\frac{dx_1}{x_1}dx_2\nonumber\\
&\geq\int_{w\not=0}\frac{F(x, u^{+}_n)}{|u^{+}_n|^{2}}|w^{+}_n|^{2}\frac{dx_1}{x_1}dx_2-\Lambda\int_{w^{+}=0}|w^{+}_n|^{2}\frac{dx_1}{x_1}dx_2.
\end{align}
For $x\in \Theta:=\{x\in \mathbb{B}: w^{+}(x)\not=0\}$, we have $|u^{+}_n(x)|\rightarrow+\infty$. By the condition ($f_{2}$) we have
\begin{equation}
    \frac{f(x, u^{+}_n)u_n}{|u^{+}_n|^{2}}|w^{+}_n|^{2}\rightarrow +\infty.
\end{equation}
Note that the Lebesgue measure of $\Theta$ is positive, using the Fatou Lemma we deduce
\begin{equation}
    \int_{w^{+}\not=0}\frac{f(x, u^{+}_n)u^{+}_n}{|u^{+}_n|^{2}}|w^{+}_n|^{2}\frac{dx_1}{x_1}dx_2\rightarrow +\infty.
\end{equation}
This contradicts with (\ref{e514}).

This proves that $\{u_n\}$ is bounded in $\mathcal{H}_{2,0}^{1,1}(\mathbb{B})$.
Without loss of generality, suppose that
\begin{equation} \label{e57}
 \left\{
\begin{array}{l}
  \|u_n\|\leq K,\\
    u_n\rightharpoonup u\ \ \mbox{in} \ \  \mathcal{H}_{2,0}^{1,1}(\mathbb{B}), \\
   u_n \rightarrow u \ \ \mbox{a.e.} \ \mathbb{B},\\
   u_n\rightarrow u \ \ \mbox{in}  \ \ L_p^{1}(\mathbb{B}), \mbox{\  for\  all\ } p>1.
   \end{array}\right.
\end{equation}
Now, since $f$ has the subcritical exponential growth on $\mathbb{B}$, we can find a constant $c_K > 0$ such
that
$$f(x,s)\leq c_Ke^{\frac{\alpha_2}{2K^2}s^2},\ \ \mbox{for\ all}\ (x,s)\in \mathbb{B}\times \mathbb{R}.$$
Then from the cone Moser-Trudinger inequality, we deduce
\begin{align}\label{e58}
  \left| \int_{\mathbb{B}} |f(x,u_n)(u_n-u) \frac{dx_1}{x_1}dx_2  \right| &\leq \int_{\mathbb{B}} |f(x,u_n)(u_n-u)| \frac{dx_1}{x_1}dx_2 \nonumber \\
   &\leq \left(\int_{\mathbb{B}} |f(x,u_n)|^2\frac{dx_1}{x_1}dx_2\right)^{\frac{1}{2}} \cdot  \left(\int_{\mathbb{B}} |u_n-u|^2 \frac{dx_1}{x_1}dx_2\right)^{\frac{1}{2}} \nonumber \\
   &\leq   C\left(\int_{\mathbb{B}}  e^{\frac{\alpha_2}{K^2}u_n^2}\frac{dx_1}{x_1}dx_2\right)^{\frac{1}{2}}\cdot  \|u_n-u\|_2\nonumber \\
   &\leq   C\left(\int_{\mathbb{B}}  e^{\frac{\alpha_2}{K^2}\|u_n\|^2\left(\frac{u_n}{\|u_n\|}\right)^2}\frac{dx_1}{x_1}dx_2\right)^{\frac{1}{2}}\cdot  \|u_n-u\|_2\nonumber \\
   &\leq C\|u_n-u\|_2\rightarrow 0,\  (n\rightarrow \infty).
\end{align}
Similarly, since $u_n\rightharpoonup u\  \mbox{in} \   \mathcal{H}_{2,0}^{1,1}(\mathbb{B})$,
$\int_{\mathbb{B}} f(x,u)(u_n-u) \frac{dx_1}{x_1}dx_2\rightarrow 0$.
 Thus we can conclude that
 \begin{equation}\label{e61}
   \int_{\mathbb{B}}  (f(x,u_n)-f(x,u)) (u_n-u) \frac{dx_1}{x_1}dx_2\rightarrow 0,\,  \text{as}\, n\rightarrow \infty.
 \end{equation}
Moreover, by \eqref{e054}
\begin{equation}\label{e60}
 (I'(u_n)-I'(u),u_n-u) \rightarrow 0,\,  \text{as}\, n\rightarrow \infty.
\end{equation}
From \eqref{e61} and \eqref{e60}, we get
$$\int_{\mathbb{B}}|\nabla_{\mathbb{B}} u_n-\nabla_{\mathbb{B}} u|^2\frac{dx_1}{x_1}dx_2\rightarrow 0,\,  \text{as}\, n\rightarrow \infty.$$
So we have $u_n\rightarrow u$  strongly in $\mathcal{H}_{2,0}^{1,1}(\mathbb{B})$
 which shows that $I$ satisfies $(PS)_c$ condition.
\end{proof}

\begin{proof}[\textbf{Proof of Theorem \ref{t41}}]
By Lemma \ref{l46}-- Lemma \ref{l48} and Mountain Pass Theorem (Proposition  \ref{mmp1}), it is clear that we can
 deduce that the problem (\ref{P}) has a nontrivial weak solution.
\end{proof}

\begin{proof}[\textbf{Proof of Theorem \ref{t42}}]
Similar to the proof of Theorem  \ref{t41}, by our conditions, we see that the functional $I$ satisfies $(C)_c$ condition.
Now we consider the Moser functions
\begin{equation}\label{e70}
 \bar{M}_2(x)=\frac{1}{\omega_1^2}
 \left\{
\begin{aligned}
  &\sqrt{\ln 2}, \ 0\leq |x|\leq \frac{1}{2},\\
    &\frac{\ln(1/|x|)}{\sqrt{\ln2}},\  \frac{1}{2}\leq |x|\leq 1, \\
  &0, \   |x|\geq 1.
   \end{aligned}\right.
\end{equation}
Obviously, $\bar{M}_2(x)\in \mathcal{H}_{2,0}^{1,1}(B_1(1,0))$ and $\|M_n\|=1$, for all $n\in \mathbb{N}$. Since $d$ is the inner radius of $\mathbb{B}$,
we can find $x_0\in \mathbb{B}$, such that $B_d(x_0)\in \mathbb{B}$. Moreover, we set $M_2(x)=\bar{M}_2\left(\frac{x-x_0}{d}\right)$. And we see  that
$M_2(x)\in \mathcal{H}_{2,0}^{1,1}(B_1(1,0))$, $\|M_2\|=1$ and $\mbox{supp}M_2=B_d(x_0)$. As in proof Theorem 1.3 in \cite{Fig}, we can deduce that
$$\max\{I(tM_2): t\geq 0\}<\frac{1}{2}\left(\frac{\alpha_2}{\alpha_0}\right).$$
It is easy to show that $I$ satisfy the  mountain pass geometry. Hence, we can find a Cerami sequence $\{u_n\}$ such that
\begin{equation}\label{e71}
I(u_n)\rightarrow C_M<\frac{1}{2}\left(\frac{\alpha_2}{\alpha_0}\right),\ \ \ \|I^{\prime}(u_n)\|\cdot\|u_n\|\rightarrow 0.
\end{equation}
We shall prove that $\{u_n\}$ is bounded in $\mathcal{H}_{2,0}^{1,1}(\mathbb{B})$. In fact, if we suppose that $\{u_n\}$ is unbounded,
let $w_n=\frac{u_n}{\|u_n\|}$,  up to a subsequence, and  we may assume that

\begin{equation}
w_n\rightharpoonup w \ \mathrm{in}\ \mathcal{H}_{2,0}^{1,1}(\mathbb{B}),\ \ w_n\rightarrow w\ \mathrm{in}\ L_p^1(\mathbb{B}),\ \ w_n\rightarrow w\ \mathrm{a.e.}\ x\in\mathbb{B}.
\end{equation}
We may similarly show that $w^{+}_n\rightarrow w^{+}$ in   $\mathcal{H}_{2,0}^{1,1}(\mathbb{B})$, where $w^{+}=\max\{w,0\}$.
Let $t_n\in [0,1]$ such that $$I(t_nu_n)=\max_{t\in[0,1]} I(tu_n),$$
and $m\in\left(0,\frac{1}{2}\left(\frac{\alpha_2}{\alpha_0}\right)^{\frac{1}{2}}\right) $. Choose $\varepsilon=\frac{\alpha_2}{m^2}-\alpha_0>0$,
according to the condition $(f_1)$, there exists $C>0$ such that
\begin{equation}\label{e74}
 F(x,s)\leq C|s|+   \left|\frac{\alpha_2}{m^2}-\alpha_0\right| e^{(\alpha_0+\varepsilon)s^2},\ \ \mbox{for\ all}\ (x,s)\in \mathbb{B}\times \mathbb{R}.
\end{equation}
Since $\|u_n\|\rightarrow \infty$, we deduce
\begin{equation}\label{e72}
 I(t_nu_n)\geq I\left(\frac{m}{\|u_n\|}u_n\right)=I(mw_n),
\end{equation}
and by (\ref{e74}) and  $\|w_n\|=1$, it follows that
\begin{align}\label{e73}
  2I(mw_n) & \geq  m^2-2cm\int_{\mathbb{B}}|w_{n}^{+}| \frac{dx_1}{x_1}dx_2-2 \left|\frac{\alpha_2}{m^2}-\alpha_0\right|\int_{\mathbb{B}}
   e^{(\alpha_0+\varepsilon) m^2 w_n^2}\frac{dx_1}{x_1}dx_2.
\end{align}
From the cone  Moser-Trudinger inequality (Lemma 2.3), we know that
$$\int_{\mathbb{B}}e^{(\alpha_0+\varepsilon) m^2 w_n^2}\frac{dx_1}{x_1}dx_2=
\int_{\mathbb{B}}e^{\alpha_2w_n^2}\frac{dx_1}{x_1}dx_2$$
is bounded by an universal constant $C (\mathbb{B}) > 0$ thanks to the choice of $\varepsilon$.
Also, since $w_n^{+}\rightharpoonup 0$ in  $\mathcal{H}_{2,0}^{1,1}(\mathbb{B})$,
we have that
$\int_{\mathbb{B}}|w_{n}^{+}| \frac{dx_1}{x_1}dx_2\rightarrow 0$. Thus if we let $n\rightarrow \infty$ in  \eqref{e73},
and then let $m\rightarrow \left[\left(\frac{\alpha_2}{\alpha_0}\right)^{\frac{1}{2}}\right]^{-}$ and using \eqref{e72}, we obtain
\begin{equation}\label{e75}
  \liminf_{n\rightarrow \infty} I(t_nu_n)\geq\frac{1}{2} \left(\frac{\alpha_2}{\alpha_0}\right)>C_M.
\end{equation}
Now note that $I(0)=0$ and $I(u_n)\rightarrow C_M$, we can assume that $t_n\in (0,1)$.  Since $I'(t_nu_n)t_nu_n=0,$ we get
$$t_n^2\|u_n\|^2=\int_{\mathbb{B}} f(x,t_nu_n) t_nu_n\frac{dx_1}{x_1}dx_2.$$
Also, \eqref{e71} implies that
$$\int_{\mathbb{B}} [f(x,u_n)u_n-2F(x,u_n)]\frac{dx_1}{x_1}dx_2=\|u_n\|^2+2C_M-\|u_n\|^2+o(1)=2C_M+o(1).$$
According to the condition $(f_3)$, we know that
\begin{align}\label{e76}
  2I(t_nu_n)&=t_n^2\|u_n\|^2-\int_{\mathbb{B}} 2F(x,t_nu_n)\frac{dx_1}{x_1}dx_2 \nonumber \\
   &=\int_{\mathbb{B}} [f(x,t_nu_n)t_nu_n-2F(x,t_nu_n)]\frac{dx_1}{x_1}dx_2\nonumber \\
   &\leq \int_{\mathbb{B}} [f(x,u_n)u_n-2F(x,u_n)]\frac{dx_1}{x_1}dx_2\nonumber \\
   &= 2C_M+o(1),
\end{align}
which  contradicts with \eqref{e75}.    Therefore, $\{u_n\}$ is bounded in  $\mathcal{H}_{2,0}^{1,1}(\mathbb{B})$.
Then, up to a subsequence,
we can suppose that $u_n \rightharpoonup u$ in $\mathcal{H}_{2,0}^{1,1}(\mathbb{B})$.
Now, following the proof of Lemma 4 in \cite{do}, we
know that $u$ is a weak solution of (\ref{P}). So  we only need to show that  $u\not=0$.
Indeed, if  $u=0$,
as in \cite{do}, we have $f(x,u_n)\rightarrow 0$ in $L^1_1(\mathbb{B})$.
The condition $(f_{6})$ implies that  $F(x, u_n) \rightarrow0 $ in
$L_1^1(\mathbb{B})$ and we
can get
\begin{equation}\label{e77}
\lim_{n\rightarrow \infty}\|u_n\|^2=2C_M<\frac{\alpha_2}{\alpha_0}
\end{equation}
and again, following the proof in \cite{do}, we have a contradiction.

 The proof is completed.

\end{proof}

\section*{Acknowledgements.}
The first author is supported by  Beijing Municipal Natural Science Foundation  (No. 1172005) and  NSFC (No. 11626038). 
 The second author is supported by NSFC
(No. 11301181) and China Postdoctoral Science Foundation.


\begin{thebibliography}{99}
\addtolength{\itemsep}{-1.5 em} 
\setlength{\itemsep}{-5pt}

\bibitem{Adac}
S. Adachi,  K.Tanaka, Trudinger type inequalities in $\mathbb{R}^N$ and their best exponents. Proc. Am. Math. Soc. 128(1999) 2051--2057.



\bibitem{AF} R.A. Adams, John J.F. Fournier, Sobolev Spaces, second edition, Elsevier, Academic Press, Amsterdam, 2003.



\bibitem{Ada}
D. R. Adams, A sharp inequality of J. Moser for higher order derivatives, Ann. of Math. 128 (1988) 385--398.




\bibitem{Alim}
M.Alimohammady, C.Cattanib, M. Koozehgar Kalleji, Invariance and existence analysis for semilinear hyperbolic equations with damping and conical singularity,
J. Math. Anal. Appl, https://doi.org/10.1016/j.jmaa.2017.05.057.





\bibitem{MR0370183}
A. Ambrosetti, P. H. Rabinowitz, Dual variational methods in critical point theory and applications,  J. Functional Analysis 14 ( 1973) 349--381.


\bibitem{Cao}
D. Cao, Nontrivial solution of semilinear elliptic equations with critical exponent in $R^2$, Comm. Partial Differential
Equations 17 (1992) 407--435.









\bibitem{Car}
L. Carleson, S. Y. A. Chang, On the existence of an extremal function for an inequality of J. Moser, Bull. Sc. Math. 110 (1986), 113--127.

\bibitem{zhang}
A. Cianchi, E. Lutwak, D. Yang, G.Y. Zhang, Affine Moser-Trudinger and Morrey-Sobolev inequalities, Calc. Var. Partial Differential Equations 36(3)(2009) 419--436.

\bibitem{chang}
S.Y.A. Chang, P. Yang, The inequality of Moser and Trudinger and applications to conformal geometry. Comm. Pure Appl. Math. 56 (2003) 1135--1150.


\bibitem{CLiu}

H. Chen, G. Liu, Global existence and nonexistence for semilinear parabolic equations with conical degeneration, J.
Pseudo-Differ. Oper. Appl. 3 (3) (2012) 329--349.


\bibitem{CLW}
H. Chen, X. Liu, Y. Wei, Cone Sobolev inequality and Dirichlet problem for nonlinear elliptic equations on manifold with
conical singularites, Calc. Var. Partial Differential Equations  43 (2012) 463--484.



\bibitem{CLW2}
H. Chen, X. Liu, Y. Wei, Existence theorem for a class of semilinear totally characteristic elliptic equations with critical cone
Sobolev exponents, Ann. Global Anal. Geom. 39 (1) (2011) 27--43.

\bibitem{Chen3}
H. Chen, X.C. Liu, Y.W. Wei,  Multiple solutions for semilinear totally characteristic elliptic equations with subcritical or critical cone Sobolev exponents,
J. Differential Equations 252 (2012) 4200--4228.


\bibitem{css}
S. Coriasco, E. Schrohe, J. Seiler,  Realizations of differential operators on conic manifolds with boundary,
Ann. Glob. Anal. Geom. 31(2007) 223--285.

\bibitem{Fig}
D. G. de Figueiredo, O. H. Miyagaki, B. Ruf, Elliptic equations in $\mathbb{R}^2$ with nonlinearities in the critical
growth range, Calc. Var. Partial Differential Equations 3(2) (1995) 139--153.


\bibitem{do}
J.M. do \'{O}, Semilinear Dirichlet problems for the $N$-Laplacian in $\mathbb{R}^N$ with nonlinearities in the critical
growth range, Differential Integral Equations 9  (5)(1996) 967--979.



\bibitem{es}
 Ju.V. Egorov,  B.W. Schulze, Pseudo-differential operators, singularities, applications, In: Operator
Theory, Advances and Applications, vol. 93. Birkh\"{a}user Verlag, Basel (1997).

\bibitem{Flu}
M. Flucher, Extremal functions for the Trudinger-Moser inequality in 2 dimensions, Comm. Math. Helvetici 67 (1992) 471--497.


\bibitem{MR1718530}
L. Jeanjean, On the existence of bounded {P}alais-{S}male sequences and
  application to a {L}andesman-{L}azer-type problem set on $\mathbb{R}^N$, Proc.
  Roy. Soc. Edinburgh Sect. A 129~(4) (1999) 787--809.

\bibitem{luguozhen}
N. Lam,  G.Z.  Lu,  H.L. Tang,  Sharp Affine and Improved Moser-Trudinger-Adams Type Inequalities on Unbounded Domains in the Spirit of Lions, J. Geom. Anal.
27 (2017) 300--334.

\bibitem{LamLu}
N. Lam, G.Z.  Lu, $N$-Laplacian equations in $\mathbb{R}^N$ with subcritical and critical growth without the Ambrosetti-Rabinowitz condition, Adv. Nonlinear Stud. 13 (2) (2013) 289-308.





\bibitem{Len}
G. Leoni, A First Course in Sobolev Spaces，Graduate Studies in Mathematics Volume 105, American Mathematical Society Providence, Rhode Island, 2009.



\bibitem{Lixy}
Y.X. Li, Extremal functions for the Moser-Trudinger inequalities on compact Riemannian manifolds, Science in China Series A: Mathematics 48 (2005) 618--648.



\bibitem{Lions}
P.L. Lions,  The concentration-compactness principle in the calculus of variations, The limit case. II.
Rev. Mat. Iberoamericana 1(2)(1985) 45--121




\bibitem{mali}
L. Ma, B.W. Schulze, Operators on manifolds with conical singularities, J. Pseudo-Differ. Oper. Appl. 1 (2010) 55--74.




\bibitem{Mcl}
J. B. McLeod, L. A. Peletier, Observations on Moser's inequality, Arch. Rat. Mech. Anal. 106 (1989) 261--285.


\bibitem{Mo}
J. Moser, A sharp form of an inequality by N. Trudinger, Indiana Univ. Math. J. 20 (1979) 1077-1092.

\bibitem{Oga}
T. Ogawa, A proof of Trudinger's inequality and its application to nonlinear Schr\"{o}dinger equations, Nonlinear Anal. 14 (1990) 765--769.

\bibitem{Oza}
T. Ozawa, On critical cases of Sobolev's inequalities, J. Funct. Anal. 127 (1995) 259--269.


\bibitem{ss}
E. Schrohe, J. Seiler, Ellipticity and invertibility in the cone algebra on $L^{p}$-Sobolev spaces, Integr. Equ.
Oper. Theory 41 (2001) 93--114.


\bibitem{sb}
B.W. Schulze,  Boundary value problems and singular pseudo-differential operators. Wiley, Chichester (1998).


\bibitem{Stri}
R. S. Strichartz, A note on Trudinger's extension of Sobolev's inequalities, Indiana Univ. Math. J. 21 (1972) 841--842.

\bibitem{Str}
M. Struwe, Critical points of embeddings of  $H_0^{1,n}$ into Orlicz spaces, Ann. Inst. Henri Poincar\'{e}, Analyse non lin\'{e}aire 5 (1988) 425--464.

\bibitem{tian}
G.Tian,X.H. Zhu, A nonlinear inequality of Moser-Trudinger type. Calc. Var. Partial Differ. Equ. 10(4)(2000) 349--354.


\bibitem{Tr}
N. S. Trudinger, On imbeddings into Orlicz spaces and some applications, J. Math. Mech.
17 (1967) 473--484.



\bibitem{Yang}
Y.Y. Yang Trudinger-Moser inequalities on complete noncompact Riemannian manifolds,  J. Functional Analysis 263 (2012) 1894--1938.


\bibitem{Pohozaev}
S. Pohozaev, The Sobolev embedding in the special case $pl = n$, in: Proceedings of the Technical Scientific Conference
on Advances of Scientific Research 1964-1965, Mathematics Sections, Moscov. Energet. Inst., Moscow,
(1965) 158--170.




\bibitem{MR1828880}
W. M. Zou, Variant fountain theorems and their applications, Manuscripta Math.
  104~(3) (2001) 343--358.

\end{thebibliography}
\end{document}